\documentclass[twoside,english,american,DIV12,listtotoc,bibtotoc,idxtotoc]{scrartcl}
\usepackage[T1]{fontenc}
\usepackage[latin1]{inputenc}
\usepackage{babel}
\usepackage{prettyref}
\usepackage{amsthm}
\usepackage{amsmath}
\usepackage{amssymb}
\usepackage[unicode=true,pdfusetitle,
 bookmarks=true,bookmarksnumbered=false,bookmarksopen=false,
 breaklinks=false,pdfborder={0 0 1},backref=false,colorlinks=false]
 {hyperref}

\makeatletter
\theoremstyle{plain}
\newtheorem{thm}{\protect\theoremname}[section]
  \theoremstyle{remark}
  \newtheorem{rem}[thm]{\protect\remarkname}
  \theoremstyle{definition}
  \newtheorem{defn}[thm]{\protect\definitionname}
  \theoremstyle{plain}
  \newtheorem{lem}[thm]{\protect\lemmaname}
  \theoremstyle{definition}
  \newtheorem{example}[thm]{\protect\examplename}

\usepackage{helvet}
\usepackage[T1]{fontenc}

\usepackage{amsmath}
\usepackage{setspace}
\usepackage{amssymb}
\usepackage{enumerate}
\usepackage{url}
\usepackage{prettyref}

\newrefformat{prop}{Proposition \ref{#1}}
\newrefformat{lem}{Lemma \ref{#1}}
\newrefformat{thm}{Theorem \ref{#1}}
\newrefformat{cor}{Corollary \ref{#1}}
\newrefformat{rem}{Remark \ref{#1}}
\newrefformat{eq}{(\ref{#1})}
\newrefformat{item}{(\ref{#1})}

\makeatletter

\theoremstyle{definition}
\newtheorem*{hyp}{Hypotheses}

\usepackage[T1]{fontenc}    

\usepackage{a4wide}        
\usepackage{amsmath}

\usepackage{euscript}
\usepackage{mathtools}
\allowdisplaybreaks[1] 
\usepackage{amsfonts}
\newenvironment{keywords}{ \noindent\footnotesize\textbf{Keywords and phrases:}}{}

\newenvironment{class}{\noindent\footnotesize\textbf{Mathematics subject classification 2010:}}{}

\newcommand*{\dive}{\operatorname{div}}

\newcommand*{\grad}{\operatorname{grad}}

\newcommand*{\spt}{\operatorname{spt}}

\renewcommand*{\i}{\mathrm{i}}

\DeclareMathAccent{\Circ}{\mathalpha}{operators}{"17}
\newcommand{\interior}[1]{\Circ{#1}}
\renewcommand{\Im}{\operatorname{\mathfrak{Im}}}
\renewcommand{\Re}{\operatorname{\mathfrak{Re}}}

\renewcommand{\hat}{\widehat}
\renewcommand{\tilde}{\widetilde}
\renewcommand*{\epsilon}{\varepsilon}
\renewcommand*{\rho}{\varrho}

\makeatother

\usepackage{babel}
\AtBeginDocument{
  
}

\makeatother

  \addto\captionsamerican{\renewcommand{\definitionname}{Definition}}
  \addto\captionsamerican{\renewcommand{\examplename}{Example}}
  \addto\captionsamerican{\renewcommand{\lemmaname}{Lemma}}
  \addto\captionsamerican{\renewcommand{\remarkname}{Remark}}
  \addto\captionsamerican{\renewcommand{\theoremname}{Theorem}}
  \addto\captionsenglish{\renewcommand{\definitionname}{Definition}}
  \addto\captionsenglish{\renewcommand{\examplename}{Example}}
  \addto\captionsenglish{\renewcommand{\lemmaname}{Lemma}}
  \addto\captionsenglish{\renewcommand{\remarkname}{Remark}}
  \addto\captionsenglish{\renewcommand{\theoremname}{Theorem}}
  \providecommand{\definitionname}{Definition}
  \providecommand{\examplename}{Example}
  \providecommand{\lemmaname}{Lemma}
  \providecommand{\remarkname}{Remark}
\providecommand{\theoremname}{Theorem}

\begin{document}

\author{ Sascha Trostorff \\ Institut f\"ur Analysis, Fachrichtung Mathematik\\ Technische Universit\"at Dresden\\ Germany\\ sascha.trostorff@tu-dresden.de}

\title{Exponential Stability for Linear Evolutionary Equations.}

\maketitle
\begin{abstract} \textbf{Abstract.} We give an approach to exponential
stability within the framework of evolutionary equations due to {[}R.
Picard. A structural observation for linear material laws in classical
mathematical physics. Math. Methods Appl. Sci., 32(14):1768\textendash{}1803,
2009{]}. We derive sufficient conditions for exponential stability
in terms of the material law operator which is defined via an analytic
and bounded operator-valued function and give an estimate for the
expected decay rate. The results are illustrated by three examples:
differential-algebraic equations, partial differential equations with
finite delay and parabolic integro-differential equations.\end{abstract}


\begin{class} 35B35 (Stability); 35B40 (Asymptotic behavior of solutions);
47N20 (Applications to differential and integral equations); 35F16
(Initial-boundary value problems for linear first-order equations)
 \end{class}

\newpage

\tableofcontents{} 

\newpage

\section{Introduction}

The article gives an approach to the exponential stability of equations
of the form 
\begin{equation}
\partial_{0}V+AU=F,\label{eq:1}
\end{equation}
where we denote the derivative with respect to the temporal variable
by $\partial_{0}$. The (unbounded) linear operator $A$, acting on
some Hilbert space, is assumed to be maximal monotone and we refer
to the monographs \cite{Brezis,Morosanu} for the topic of monotone
operators on Hilbert spaces. Equation \prettyref{eq:1} is completed
by a constitutive relation of the form 
\[
V=\mathcal{M}U,
\]
where $\mathcal{M}$ is a bounded linear operator acting in time and
space. Thus, we end up with an equation of the form 
\[
\left(\partial_{0}\mathcal{M}+A\right)U=F,
\]
which we refer to as an \emph{evolutionary equation. }This class of
problems was introduced by Picard in \cite{Picard}, where $A$ was
assumed to be skew-selfadjoint, and it was illustrated that many equations
of classical mathematical physics are covered by this abstract class.
The well-posedness for such problems is proved, by showing that the
operator $\partial_{0}\mathcal{M}+A$ is boundedly invertible in a
suitable Hilbert space. More precisely, the derivative $\partial_{0}$
is established as a normal, continuously invertible operator in an
exponentially weighted $L_{2}$-space and $\mathcal{M}$ is defined
as a function of $\partial_{0}^{-1}$ (see Section 2) and the well-posedness
is shown under a positive definiteness constraint on the operator
$\mathcal{M}$ (see \cite[Solution Theory]{Picard} and \prettyref{thm:sol_theo}
of this article). Moreover, the question of causality, which can be
seen as a characterizing property for evolutionary processes, was
addressed which leads to additional constraints on the operator $\mathcal{M},$
namely that $\mathcal{M}$ is defined via the Fourier-Laplace transformation
of an analytic and bounded function $M:B_{\mathbb{C}}(r,r)\to L(H)$
(for more details see Section 2). Especially the analyticity of $M$
is crucial for the causality, due to the correlation of supports of
$L_{2}$-functions and the analyticity of their Laplace transforms
by the Paley-Wiener Theorem (see \cite[Theorem 19.2]{rudin1987real}).
Later on these results were generalized to the case of $A$ being
a maximal monotone relation in \cite{Trostorff2012_NA,Trostorff2012_nonlin_bd}. 

In this work we give sufficient criteria for the exponential stability
of the evolutionary problem in terms of the function $M$. The study
of stability issues for differential equations, which goes back to
Lyapunov (see \cite{ljapunov1992general} for a survey), has become
a very active field of research for many decades and there exist numerous
works dealing with this topic. We just like to mention some standard
approaches to exponential stability. The first strategy goes back
to Lyapunov. The aim is to find a suitable function (a so-called Lyapunov
function) yielding a certain differential inequality which allows
to derive statements about the asymptotic behavior of solutions of
the differential equation.  A second approach, which applies to linear
differential equations, is based on the theory of semigroups. In this
framework different criteria for exponential stability were derived
in terms of the semigroup or its generator, e.g. Datko's Lemma (\cite{Datko_1970}
or \cite[p. 300]{engel2000one}), Gearharts Theorem (\cite{Gearhart_1978}
or \cite[p. 302]{engel2000one}) or the Spectral Mapping Theorem (see
\cite[p. 302, Theorem 1.10]{engel2000one}). A third approach uses
the Fourier or the Laplace transform of a solution to derive statements
of their asymptotics. These methods are sometimes referred to as Frequency
Domain Methods. In our framework it seems to be appropriate to employ
the last approach, since, by the definition of $\mathcal{M}$, methods
of vector-valued complex analysis are at hand through the Fourier-Laplace
transformation. Note that, due to the general structure of evolutionary
equations, the results apply to a broad class of differential equations,
such as differential-algebraic equations, equations with memory effects
or integro-differential equations, where semigroup methods may be
difficult to apply. 

The article is structured as follows. In Section 2 we recall the framework
of evolutionary equations and its solution theory (\prettyref{thm:sol_theo}).
Section 3 provides an abstract condition for exponential stability
for evolutionary equations in terms of the function $M$ and the proof
of the main stability result (\prettyref{thm:stability}). To illustrate
the versatility of the previous results, we discuss several examples
in Section 4. We begin with studying differential-algebraic equations
of mixed type and derive a condition for their exponential stability.
Moreover, we give a concrete example for such an equation, which seems
hard to be tackled by other approaches, since the type of the differential
equation switches on different parts of the underlying domain. Furthermore,
we give a possible approach of how to deal with initial value problems.
In Subsection 4.2 we consider an example of a partial differential
equation with memory effect. A similar problem was also treated by
Batkai and Piazerra in \cite{Batkai_2005}, using a semigroup approach.
In contrast to their result our approach directly extends to the case
of differential-algebraic equations with delay. We conclude the article
by discussing a parabolic integro-differential equation with an operator-valued
kernel, where we adopt the ideas of \cite{Trostorff2012_integro}
to derive the exponential stability. 

Throughout let $H$ be a complex Hilbert space. We denote its inner
product by $\langle\cdot|\cdot\rangle$ which is assumed to be linear
in the second and conjugate linear in the first argument. We denote
its induced norm by $|\cdot|.$

\section{The framework of evolutionary equations}

We recall some basic notions and results on linear evolutionary equations,
i.e. equations of the form
\begin{equation}
\left(\partial_{0,\rho}M(\partial_{0,\rho}^{-1})+A\right)u=f,\label{eq:evo}
\end{equation}

where $A:D(A)\subseteq H\to H$ is a linear, maximal monotone operator,
$\partial_{0,\rho}$ denotes the time-derivative, established in a
suitable Hilbert space and $M(\partial_{0,\rho}^{-1})$ is a bounded
linear operator in time and space, a so-called \emph{linear material
law.} We refer to \cite{Picard,Picard_McGhee,Kalauch2011,Trostorff2012_NA,Trostorff2012_nonlin_bd}
for more details and proofs of the following statements. First we
begin by introducing the Hilbert space setting, where we want to consider
equation \prettyref{eq:evo}.\\
For $\rho\in\mathbb{R}$ we define the space $H_{\rho,0}(\mathbb{R};H)$
as the space of (equivalence classes of) measurable function $f:\mathbb{R}\to H$
which are square-integrable with respect to the exponentially weighted
Lebesgue measure $e^{-2\rho t}\mbox{ d}t,$ equipped with the inner
product 
\[
\langle f|g\rangle_{H_{\rho,0}(\mathbb{R};H)}\coloneqq\intop_{\mathbb{R}}\langle f(t)|g(t)\rangle e^{-2\rho t}\mbox{ d}t.
\]
Note that $H_{0,0}(\mathbb{R};H)$ is just the space $L_{2}(\mathbb{R};H).$
We define the derivative $\partial_{0,\rho}$ as the closure of the
operator 
\begin{align*}
\partial_{0,\rho}|_{C_{c}^{\infty}(\mathbb{R};H)}:C_{c}^{\infty}(\mathbb{R};H)\subseteq H_{\rho,0}(\mathbb{R};H) & \to H_{\rho,0}(\mathbb{R};H)\\
\phi & \mapsto\phi',
\end{align*}

where we denote by $C_{c}^{\infty}(\mathbb{R};H)$ the space of function
$\phi:\mathbb{R}\to H$ which are arbitrarily often differentiable
and have compact support. 
\begin{rem}
\label{rem:partial_0}$\,$

\begin{enumerate}[(a)]

\item For each $\rho\in\mathbb{R}$ the operator $\partial_{0,\rho}$
is normal with $\Re\partial_{0,\rho}=\frac{1}{2}\left(\overline{\partial_{0,\rho}+\partial_{0,\rho}^{\ast}}\right)=\rho.$
Moreover we obtain that $\partial_{0,\rho}^{\ast}=-\partial_{0,\rho}+2\rho.$
In particular, for $\rho=0$ the operator $\partial_{0,0}$ is skew-selfadjoint
and coincides with the usual weak derivative on $L_{2}(\mathbb{R};H)$
with domain $H_{0,1}(\mathbb{R};H)=H^{1}(\mathbb{R};H)=W_{2}^{1}(\mathbb{R};H)$.

\item For $\rho\ne0$ the operators $\partial_{0,\rho}$ on $H_{\rho,0}(\mathbb{R};H)$
and $\partial_{0,0}+\rho$ on $L_{2}(\mathbb{R};H)$ are unitarily
equivalent and $\partial_{0,\rho}$ is boundedly invertible with $\|\partial_{0,\rho}^{-1}\|=\frac{1}{|\rho|}.$
For $\rho=0$ the operator $\partial_{0,0}$ is not boundedly invertible,
since $0$ lies in its continuous spectrum. 

\item For $\rho\ne0$ the inverse operator $\partial_{0,\rho}^{-1}$
is given by 
\[
\left(\partial_{0,\rho}^{-1}u\right)(t)=\begin{cases}
\intop_{-\infty}^{t}u(s)\mbox{ d}s & \mbox{ if }\rho>0,\\
-\intop_{t}^{\infty}u(s)\mbox{ d}s & \mbox{ if }\rho<0
\end{cases}
\]
for $u\in H_{\rho,0}(\mathbb{R};H)$ and almost every $t\in\mathbb{R}.$
This representation yields that for $\rho>0$ the operator $\partial_{0,\rho}^{-1}$
is \emph{forward causal} whereas it is \emph{backward causal}%
\footnote{A mapping $F:D(F)\subseteq H_{\rho,0}(\mathbb{R};H)\to H_{\rho,0}(\mathbb{R};H)$
is called \emph{forward causal} if for each $f,g\in D(F)$ with $f=g$
on some interval $]-\infty,a[$ for $a\in\mathbb{R}$ the functions
$F(f)$ and $F(g)$ coincide on the same interval $]-\infty,a[.$
Analogously $F$ is called \emph{backward causal} if for each $f,g\in D(F)$
with $f=g$ on some interval $]a,\infty[$ for $a\in\mathbb{R}$ the
functions $F(f)$ and $F(g)$ coincide on the same interval $]a,\infty[.$%
} for $\rho<0.$

\item Let $\rho\in\mathbb{R}$ and denote by $\mathcal{L}_{\rho}:H_{\rho,0}(\mathbb{R};H)\to L_{2}(\mathbb{R};H)$
the so-called \emph{Fourier-Laplace transformation, }defined as the
unitary extension of the operator given by 
\[
\left(\mathcal{L}_{\rho}\phi\right)(x)=\frac{1}{\sqrt{2\pi}}\intop_{\mathbb{R}}e^{-\i xt}e^{-\rho t}\phi(t)\mbox{ d}t
\]

for $\phi\in C_{c}^{\infty}(\mathbb{R};H)$ and $x\in\mathbb{R}$.
Then 
\begin{equation}
\partial_{0,\rho}=\mathcal{L}_{\rho}^{\ast}(\i m+\rho)\mathcal{L}_{\rho},\label{eq:spec_partial_0}
\end{equation}

where by $m:D(m)\subseteq L_{2}(\mathbb{R};H)\to L_{2}(\mathbb{R};H)$
we denote the multiplication-by-the-argument operator with maximal
domain, i.e., $\left(mf\right)(t)=tf(t)$ for almost every $t\in\mathbb{R}$
and every $f\in D(m)\coloneqq\left\{ f\in L_{2}(\mathbb{R};H)\,\left|\,\left(t\mapsto tf(t)\right)\in L_{2}(\mathbb{R};H)\right.\right\} .$
Note that in the case $\rho=0$, \prettyref{eq:spec_partial_0} is
just the usual spectral representation for the weak derivative on
$L_{2}(\mathbb{R};H)$ via the Fourier transformation (see \cite[p. 112]{Akhiezer_Glazman_1993}).

\end{enumerate}
\end{rem}
We consider \prettyref{eq:evo} as an equation in the Hilbert space
$H_{\rho,0}(\mathbb{R};H).$ As a matter of physical relevance, we
require that the corresponding solution operator $\left(\partial_{0,\rho}M(\partial_{0,\rho}^{-1})+A\right)^{-1}$,
if it exists, is forward causal, that means, roughly speaking, that
the present state of the solution $u$ does not depend on the future
behavior of the source term $f.$ With \prettyref{rem:partial_0}
(c) in mind, we therefore assume that $\rho>0.$\\
We now define the operator $M(\partial_{0,\rho}^{-1})$ for $\rho>0$
with the help of formula \prettyref{eq:spec_partial_0}.
\begin{defn}
Let $r>0$ and $M:B_{\mathbb{C}}(r,r)\to L(H)$ %
\footnote{We denote by $B_{\mathbb{C}}(x,s)$ the open ball in $\mathbb{C}$
with center $x\in\mathbb{C}$ and radius $s>0$%
}. For $\rho>\frac{1}{2r}$ we define 
\[
M(\partial_{0,\rho}^{-1})\coloneqq\mathcal{L}_{\rho}^{\ast}M\left(\frac{1}{\i m+\rho}\right)\mathcal{L}_{\rho},
\]
where $M\left(\frac{1}{\i m+\rho}\right)$ is defined as the multiplication
operator $\left(M\left(\frac{1}{\i m+\rho}\right)f\right)(t)\coloneqq M\left(\frac{1}{\i t+\rho}\right)f(t)$
on $L_{2}(\mathbb{R};H)$ with domain $\left\{ f\in L_{2}(\mathbb{R};H)\,\left|\,\left(t\mapsto M\left(\frac{1}{\i t+\rho}\right)f(t)\right)\in L_{2}(\mathbb{R};H)\right.\right\} $.
We call $M(\partial_{0,\rho}^{-1})$ a \emph{linear material law}
if the function $M$ belongs to $\mathcal{H}^{\infty}(B_{\mathbb{C}}(r,r);L(H)),$
i.e., $M$ is bounded and analytic.\emph{ }\end{defn}
\begin{rem}
The notion material law is motivated by several examples of mathematical
physics, since it turns out that all material parameters, such as
mass density, conductivity, permeability etc.\ can be incorporated
into the operator $M(\partial_{0,\rho}^{-1})$ (see \cite{Picard}
for several examples). Thus, it is natural that $M(\partial_{0,\rho}^{-1})$
is forward causal. Since the operator $M(\partial_{0,\rho}^{-1})$
is linear and it commutes with the translation operators $\tau_{h}$,
mapping $u\in H_{\rho,0}(\mathbb{R};H)$ to $t\mapsto u(t+h)$ for
$h\in\mathbb{R}$, causality can be characterized via the requirement
that $\spt M(\partial_{0,\rho}^{-1})u\subseteq]0,\infty[$ if $\spt u\subseteq]0,\infty[$,
where by $\spt g$ we denote the support of a function $g\in L_{2,\mathrm{loc}}(\mathbb{R};H).$
This, however, can be characterized by the analyticity and boundedness
of the mapping $M$, employing a Paley-Wiener-type result (see e.g.
\cite[Theorem 19.2]{rudin1987real}). Moreover, note that due to the
boundedness of $M$, the operator $M\left(\partial_{0,\rho}^{-1}\right)$
becomes a bounded operator on $H_{\rho,0}(\mathbb{R};H).$
\end{rem}
We are now able to state the solution theory for evolutionary equations.
For the proof we refer to \cite{Picard,Trostorff2012_nonlin_bd}.
\begin{thm}[Solution Theory]
\label{thm:sol_theo} Let $A:D(A)\subseteq H\to H$ be a maximal
monotone linear operator. Moreover, let $r>0$ and $M\in\mathcal{H}^{\infty}(B_{\mathbb{C}}(r,r);L(H))$
and assume that the \emph{solvability condition} is satisfied:
\begin{equation}
\exists c>0\:\forall z\in B_{\mathbb{C}}(r,r)\,,x\in H:\,\Re\langle z^{-1}M(z)x|x\rangle\geq c|x|^{2}.\label{eq:solv}
\end{equation}
Then for each $\rho>\frac{1}{2r}$ the problem \prettyref{eq:evo}
is \emph{well-posed} in $H_{\rho,0}(\mathbb{R};H)$, i.e. for each
$\rho>\frac{1}{2r}$ the operator $\partial_{0,\rho}M(\partial_{0,\rho}^{-1})+A$
is boundedly invertible and has a dense range. Moreover, the solution
operator $\left(\overline{\partial_{0,\rho}M(\partial_{0,\rho}^{-1})+A}\right)^{-1}$
is forward causal. \end{thm}
\begin{rem}
\label{rem: regularity}Note that the solution operator $\left(\overline{\partial_{0,\rho}M(\partial_{0,\rho}^{-1})+A}\right)^{-1}$
commutes with the time-derivative $\partial_{0,\rho}.$ This yields,
that for right hand sides $f\in D(\partial_{0,\rho}^{k})$ for $k\in\mathbb{N}$
in \prettyref{eq:evo} the corresponding solution $u$ also lies in
$D(\partial_{0,\rho}^{k}),$ i.e. the solution operator preserves
``temporal'' regularity%
\footnote{Indeed, one can show that the solution operator $\left(\overline{\partial_{0,\rho}M(\partial_{0,\rho}^{-1})+A}\right)^{-1}$
extends to the whole Sobolev-chain $\left(H_{k}(\partial_{0,\rho})\right)_{k\in\mathbb{Z}}$
associated with the derivative $\partial_{0,\rho}$, see \cite{Picard_McGhee}
which yields a solution theory for distributional right-hand sides.%
}.
\end{rem}

\section{An abstract condition for exponential stability}

In this section we show that under certain constraints on the function
$M$, the corresponding evolutionary problem is exponentially stable.
Although in the literature, exponential stability usually means that
the solution of an initial value problem decays with an exponential
rate as time tends to infinity, we like to introduce a slightly weaker
notion within our framework which, however, yields the desired decay
if the source term $f$ is regular enough (see \prettyref{rem: interpol}
(a)). Moreover, since the framework of evolutionary equations covers
different types of equations, where initial values do not make sense,
such as differential-algebraic or even pure algebraic equations or
equations with memory effect, where a given pre-history would be more
appropriate then an initial value, we cannot treat initial value problems
within this general approach. However, in concrete examples we can
reformulate initial value problems as evolutionary equations with
a modified right hand side (see Subsection 4.1) such that the following
results still apply. 
\begin{defn}
Let $A:D(A)\subseteq H\to H$ be a maximal monotone linear operator
and $M\in\mathcal{H}^{\infty}(B_{\mathbb{C}}(r,r);L(H))$ for some
$r>0$ which satisfies the solvability condition \prettyref{eq:solv}.
Let $\rho>\frac{1}{2r}.$ We call the solution operator $\left(\overline{\partial_{0,\rho}M(\partial_{0,\rho}^{-1})+A}\right)^{-1}$
of \prettyref{eq:evo} \emph{exponentially stable} \emph{with stability
rate }$\nu_{0}>0$ if for each $0\leq\nu<\nu_{0}$ and $f\in H_{-\nu,0}(\mathbb{R};H)\cap H_{\rho,0}(\mathbb{R};H)$
the solution $u$ of \prettyref{eq:evo} satisfies 
\[
u=\left(\overline{\partial_{0,\rho}M(\partial_{0,\rho}^{-1})+A}\right)^{-1}f\in\bigcap_{-\nu<\mu\leq\rho}H_{\mu,0}(\mathbb{R};H),
\]
which especially implies $\intop_{\mathbb{R}}e^{2\mu t}|u(t)|^{2}\mbox{ d}t<\infty$
for all $0\leq\mu<\nu.$\end{defn}
\begin{rem}
\label{rem: interpol}$\,$

\begin{enumerate}[(a)]

\item We show that our notion of exponential stability indeed yields
an exponential decay of the solution if the given right hand side
is regular enough. For doing so, let $\rho>\frac{1}{2r}$ and $\left(\overline{\partial_{0,\rho}M(\partial_{0,\rho}^{-1})+A}\right)^{-1}$
exponentially stable with stability rate $\nu_{0}>0.$ Moreover, assume
that $f\in H_{-\nu,0}(\mathbb{R};H)\cap H_{\rho,0}(\mathbb{R};H)$
and $f\in D(\partial_{0,\rho})$ such that $\partial_{0,\rho}f\in H_{-\nu,0}(\mathbb{R};H)\cap H_{\rho,0}(\mathbb{R};H)$
for some $0<\nu<\nu_{0}$. Then $u=\left(\overline{\partial_{0,\rho}M(\partial_{0,\rho}^{-1})+A}\right)^{-1}f$
also lies in $D(\partial_{0,\rho})$ and $\partial_{0,\rho}u=\left(\overline{\partial_{0,\rho}M(\partial_{0,\rho}^{-1})+A}\right)^{-1}\partial_{0,\rho}f$
(compare \prettyref{rem: regularity}). By the assumed exponential
stability, we get that $e^{\mu m}u\in L_{2}(\mathbb{R};H)$ and $e^{\mu m}\partial_{0,\rho}u\in L_{2}(\mathbb{R};H)$
for each $0\leq\mu<\nu.$ The latter yields $e^{\mu m}u\in W_{2}^{1}(\mathbb{R};H)$.
Indeed, for $\phi\in C_{c}^{\infty}(\mathbb{R};H)$ we compute 
\begin{align*}
 & \langle\mu e^{\mu m}u+e^{\mu m}\partial_{0,\rho}u|\phi\rangle_{L_{2}(\mathbb{R};H)}\\
 & =\intop_{\mathbb{R}}\langle u(t)|\mu e^{\mu t}\phi(t)\rangle\mbox{ d}t+\intop_{\mathbb{R}}\left\langle \left.\left(\partial_{0,\rho}u\right)(t)\right|e^{\left(2\rho+\mu\right)t}\phi(t)e^{-2\rho t}\right\rangle \mbox{ d}t\\
 & =\intop_{\mathbb{R}}\langle u(t)|\mu e^{\mu t}\phi(t)\rangle\mbox{ d}t+\langle\partial_{0,\rho}u|e^{(2\rho+\mu)m}\phi\rangle_{H_{\rho,0}(\mathbb{R};H)}\\
 & =\intop_{\mathbb{R}}\langle u(t)|\mu e^{\mu t}\phi(t)\rangle\mbox{ d}t+\left\langle u\left|-\partial_{0,\rho}\left(e^{(2\rho+\mu)m}\phi\right)\right.\right\rangle _{H_{\rho,0}(\mathbb{R};H)}\\
 & \quad+\left\langle u\left|2\rho\left(e^{(2\rho+\mu)m}\phi\right)\right.\right\rangle _{H_{\rho,0}(\mathbb{R};H)}\\
 & =\intop_{\mathbb{R}}\langle u(t)|\mu e^{\mu t}\phi(t)\rangle\mbox{ d}t-\intop_{\mathbb{R}}(2\rho+\mu)\langle u(t)|e^{\left(2\rho+\mu\right)t}\phi(t)\rangle e^{-2\rho t}\mbox{ d}t\\
 & \quad-\intop_{\mathbb{R}}\langle u(t)|e^{\left(2\rho+\mu\right)t}\phi'(t)\rangle e^{-2\rho t}\mbox{ d}t+\intop_{\mathbb{R}}2\rho\langle u(t)|e^{\left(2\rho+\mu\right)t}\phi(t)\rangle e^{-2\rho t}\mbox{ d}t\\
 & =-\intop_{\mathbb{R}}\langle u(t)|e^{\left(2\rho+\mu\right)t}\phi'(t)\rangle e^{-2\rho t}\mbox{ d}t\\
 & =-\langle e^{\mu m}u|\phi'\rangle_{L_{2}(\mathbb{R};H)}.
\end{align*}
Thus, we obtain $e^{\mu t}|u(t)|\to0$ as $t$ tends to infinity for
each $0\leq\mu<\nu$ due to Sobolev's embedding theorem (see e.g.\
\cite[p. 408]{engel2000one}), i.e.\ $u$ decays exponentially with
a decay rate less than $\nu.$ 

\item If $f\in H_{\mu,0}(\mathbb{R};H)\cap H_{\rho,0}(\mathbb{R};H)$
for some $\mu,\rho\in\mathbb{R}$ with $\mu>\rho$, then $f\in H_{\nu,0}(\mathbb{R};H)$
for all $\nu\in[\rho,\mu].$ Indeed, we estimate 
\begin{align*}
\intop_{\mathbb{R}}|f(t)|^{2}e^{-2\nu t}\mbox{ d}t & =\intop_{-\infty}^{0}|f(t)|^{2}e^{-2\nu t}\mbox{ d}t+\intop_{0}^{\infty}|f(t)|^{2}e^{-2\nu t}\mbox{ d}t\\
 & \leq\intop_{-\infty}^{0}|f(t)|^{2}e^{-2\mu t}\mbox{ d}t+\intop_{0}^{\infty}|f(t)|^{2}e^{-2\rho t}\mbox{ d}t\\
 & \leq|f|_{H_{\mu,0}(\mathbb{R};H)}^{2}+|f|_{H_{\rho,0}(\mathbb{R};H)}^{2}.
\end{align*}

\end{enumerate}
\end{rem}
We now give conditions for the function $M$ and show that they yield
the well-posedness and exponential stability for the corresponding
evolutionary problem.

\begin{hyp} Let $\nu_{0}>0$. We assume that 

\begin{enumerate}[(a)]

\item $M:\mathbb{C}\setminus B_{\mathbb{C}}\left[-\frac{1}{2\nu_{0}},\frac{1}{2\nu_{0}}\right]\to L(H)$
is analytic%
\footnote{We denote by $B_{\mathbb{C}}[x,s]$ the closed ball in $\mathbb{C}$
with center $x\in\mathbb{C}$ and radius $s\geq0$.%
}, 

\item for each $r>0$ and $0\leq\nu<\nu_{0}$ the function%
\footnote{Here and further on we set $B_{\mathbb{C}}(r,r)\setminus\{0^{-1}\}\coloneqq B_{\mathbb{C}}(r,r).$%
} 
\[
B_{\mathbb{C}}(r,r)\setminus\{\nu^{-1}\}\ni z\mapsto\left(1-\nu z\right)M\left(z(1-\nu z)^{-1}\right)
\]

has a bounded and analytic extension to $B_{\mathbb{C}}(r,r)$, 

\item for every $0<\nu<\nu_{0}$ there exists $c>0$ such that for
all $z\in\mathbb{C}\setminus B_{\mathbb{C}}\left[-\frac{1}{2\nu},\frac{1}{2\nu}\right]$
\[
\Re z^{-1}M(z)\geq c.
\]

\end{enumerate}

\end{hyp}
\begin{rem}
Let $M$ satisfy the assumptions above and let $r>0.$ Then the restriction
of $M$ to $B_{\mathbb{C}}(r,r)$ is an element of $\mathcal{H}^{\infty}(B_{\mathbb{C}}(r,r);L(H)).$
Indeed, the analyticity is clear from (a) and the boundedness follows
from (b). Hence, together with (c), it follows that the problem \prettyref{eq:evo}
is well-posed for every $\rho>0$ according to \prettyref{thm:sol_theo}.
\end{rem}
We now state some auxiliary results which in particular imply that
the solution operator $\left(\overline{\partial_{0,\rho}M\left(\partial_{0,\rho}^{-1}\right)+A}\right)^{-1}$
for an evolutionary problem does not depend on the particular choice
of $\rho.$ These results can also be found in \cite[p. 429 f.]{Picard_McGhee}.
However, for sake of completeness we present them again with a slightly
modified proof.
\begin{lem}
\label{lem:shift_1}Let $\rho,\mu\in\mathbb{R}$ with $\mu>\rho$
and set $U\coloneqq\{z\in\mathbb{C}\,|\,\Re z\in[\rho,\mu]\}$. Moreover,
let $f:U\to H$ be continuous on $U$ and analytic in the interior
of $U$, such that $\intop_{\rho}^{\mu}f(\i R+s)\,\mathrm{d}s\to0$
as $R\to\pm\infty$. Then 
\[
\limsup_{R\to\infty}\left|\,\intop_{-R}^{R}f(\i t+\mu)\,\mathrm{d}t-\intop_{-R}^{R}f(\i t+\rho)\,\mathrm{d}t\right|=0.
\]
\end{lem}
\begin{proof}
According to Cauchy's integral theorem we have 
\[
\i\intop_{-R}^{R}f(\i t+\rho)\mbox{ d}t+\intop_{\rho}^{\mu}f(\i R+s)\mbox{ d}s-\left(\i\intop_{-R}^{R}f(\i t+\mu)\mbox{ d}t+\intop_{\rho}^{\mu}f(-\i R+s)\mbox{ d}s\right)=0
\]
for each $R>0.$ The assertion now follows by taking the limes superior
as $R$ tend to $\infty.$ 
\end{proof}
The next lemma shows that we can approximate a function which belongs
to two different exponentially weighted $L_{2}$-spaces by the same
sequence of test functions with respect to both topologies.
\begin{lem}
\label{lem:density} Let $\rho,\mu\in\mathbb{R}$ and $f\in H_{\rho,0}(\mathbb{R};H)\cap H_{\mu,0}(\mathbb{R};H).$
Then, for each $\varepsilon>0$ there exists $\phi\in C_{c}^{\infty}(\mathbb{R};H)$
such that 
\[
\max\left\{ |f-\phi|_{H_{\rho,0}(\mathbb{R};H)},|f-\phi|_{H_{\mu,0}(\mathbb{R};H)}\right\} \leq\varepsilon.
\]
\end{lem}
\begin{proof}
Let $\varepsilon>0.$ Then we choose $N\in\mathbb{N}$ such that $f_{N}\coloneqq f\chi_{[-N,N]}$
satisfies 
\[
\max\left\{ |f-f_{N}|_{H_{\rho,0}(\mathbb{R};H)},|f-f_{N}|_{H_{\mu,0}(\mathbb{R};H)}\right\} \leq\frac{\varepsilon}{2}.
\]
We denote by $(\psi_{k})_{k\in\mathbb{N}}\in C_{c}^{\infty}(\mathbb{R})^{\mathbb{N}}$
the Friedrichs mollifier (see e.g. \cite[Chapter C.4]{evans2010partial}).
Then, for each $k\in\mathbb{N}$ we have that $\spt\psi_{k}\ast f_{N}\subseteq[-N-1,N+1].$
Now, we choose $k_{0}$ large enough such that 
\[
|\psi_{k_{0}}\ast f_{N}-f_{N}|_{L_{2}([-N-1,N+1];H)}\leq e^{-2\max\{|\rho|,|\mu|\}(N+1)}\frac{\varepsilon}{2}.
\]
Then the function $\psi_{k_{0}}\ast f_{N}\in C_{c}^{\infty}(\mathbb{R};H)$
has the desired property.\end{proof}
\begin{lem}
\label{lem:shift_2}Let $\rho,\mu\in\mathbb{R}$ with $\mu>\rho$
and set $U\coloneqq\{z\in\mathbb{C}\,|\,\Re z\in[\rho,\mu]\}$. Moreover,
let $f\in H_{\rho,0}(\mathbb{R};H)\cap H_{\mu,0}(\mathbb{R};H)$ and
$T\in\mathcal{H}^{\infty}(\interior U;L(H))\cap C_{b}(U;L(H))$ (i.e.
$T$ is bounded and continuous on $U$ and analytic in the interior
of $U$). Then 
\[
\left(\mathcal{L}_{\rho}^{\ast}T(\i m+\rho)\mathcal{L}_{\rho}f\right)(t)=\left(\mathcal{L}_{\mu}^{\ast}T(\i m+\mu)\mathcal{L}_{\mu}f\right)(t)
\]
for almost every $t\in\mathbb{R}.$ \end{lem}
\begin{proof}
According to \prettyref{lem:density} it suffices to prove the assertion
for test functions. So let $\phi\in C_{c}^{\infty}(\mathbb{R};H)$.
We show that the function 
\[
U\ni z\mapsto e^{zt}T(z)\left(\left(\mathcal{L}_{\Re z}\phi\right)(\Im z)\right)\in H
\]
satisfies the assumptions of \prettyref{lem:shift_1}. Indeed it is
continuous in $U$ and analytic in the interior of $U$ as a composition
of analytic functions. Furthermore, we estimate for $s\in[\rho,\mu],\xi\in\mathbb{R}:$
\[
\left|\mathcal{L}_{s}\phi(\xi)\right|\leq\frac{1}{\sqrt{2\pi}}\intop_{\mathbb{R}}e^{-rs}|\phi(r)|\mbox{ d}r\leq\frac{C}{\sqrt{2\pi}}\intop_{\mathbb{R}}|\phi(r)|\mbox{ d}r,
\]
where $C\coloneqq\sup\{e^{-rs}\,|\, s\in[\rho,\mu],r\in\spt\phi\}$,
which shows that the function 
\[
U\ni z\mapsto\left(\mathcal{L}_{\Re z}\phi\right)(\Im z)
\]
is bounded. Moreover, due to the Riemann-Lebesgue lemma, we get that
$\left(\mathcal{L}_{s}\phi\right)(R)\to0$ as $R\to\pm\infty$ for
every $s\in\mathbb{R}.$ Therefore, according to Lebesgue's dominated
convergence theorem, we deduce that 
\[
\left|\intop_{\rho}^{\mu}e^{\left(\i R+s\right)t}T(\i R+s)\left(\mathcal{L}_{s}\phi\right)(R)\mbox{ d}s\right|\leq\max\{e^{\rho t},e^{\mu t}\}|T|_{\infty}\intop_{\rho}^{\mu}|(\mathcal{L}_{s}\phi)(R)|\mbox{ d}s\to0\quad(R\to\pm\infty).
\]
Thus, by \prettyref{lem:shift_1}, we get that 
\[
\intop_{\mathbb{R}}e^{\left(\i s+\rho\right)t}T(\i s+\rho)\left(\mathcal{L}_{\rho}\phi\right)(s)\mbox{ d}s=\intop_{\mathbb{R}}e^{\left(\i s+\mu\right)t}T(\i s+\mu)\mathcal{L}_{\mu}\phi(s)\mbox{ d}s,
\]
which yields the assertion. 
\end{proof}
We are now able to prove our main theorem.
\begin{thm}
\label{thm:stability}Let $A:D(A)\subseteq H\to H$ be a maximal monotone
linear operator and $M$ a mapping satisfying the hypotheses above
for some $\nu_{0}>0$. Then for each $\rho>0$ the solution operator
$\left(\overline{\partial_{0,\rho}M(\partial_{0,\rho}^{-1})+A}\right)^{-1}$
is exponentially stable with stability rate $\nu_{0}$. \end{thm}
\begin{proof}
Let $\rho>0$, $0\leq\nu<\nu_{0}$ and take $f\in H_{-\nu,0}(\mathbb{R};H)\cap H_{\rho,0}(\mathbb{R};H).$
We set 
\[
u\coloneqq\left(\overline{\partial_{0,\rho}M(\partial_{0,\rho}^{-1})+A}\right)^{-1}f
\]
and we have to show that $u\in H_{\mu,0}(\mathbb{R};H)$ for each
$\mu\in]-\nu,\rho].$ Let $0<\eta<\rho+\nu.$ We define 
\[
\tilde{N}(z)=(1-\nu z)M\left(z\left(1-\nu z\right)^{-1}\right)
\]
for $z\in B_{\mathbb{C}}\left(\frac{1}{2\eta},\frac{1}{2\eta}\right)\setminus\{\nu^{-1}\}.$
Note that according to hypotheses (b), $\tilde{N}$ has an extension
$N\in\mathcal{H}^{\infty}\left(B_{\mathbb{C}}\left(\frac{1}{2\eta},\frac{1}{2\eta}\right);L(H)\right).$
Moreover, 
\begin{equation}
\Re z^{-1}N(z)=\Re z^{-1}\tilde{N}(z)=\Re\frac{(1-\nu z)}{z}M\left(\frac{z}{1-\nu z}\right)\geq c\label{eq:pos_def}
\end{equation}
for every $z\in B_{\mathbb{C}}\left(\frac{1}{2\eta},\frac{1}{2\eta}\right)\setminus\{\nu^{-1}\}$
and some suitable $c>0,$ since $z\left(1-\nu z\right)^{-1}\notin B_{\mathbb{C}}\left[-\frac{1}{2\nu},\frac{1}{2\nu}\right]$.
Due to the continuity of $N$, estimate \prettyref{eq:pos_def} also
holds for $z=\frac{1}{\nu}$. Thus, according to \prettyref{thm:sol_theo},
we obtain a solution 
\[
v\coloneqq\left(\overline{\partial_{0,\eta}N(\partial_{0,\eta}^{-1})+A}\right)^{-1}e^{\nu m}f\in H_{\eta,0}(\mathbb{R};H),
\]
where we have used that $e^{\nu m}f\in H_{0,0}(\mathbb{R};H)\cap H_{\rho+\nu,0}(\mathbb{R};H)\subseteq H_{\eta,0}(\mathbb{R};H)$
(see \prettyref{rem: interpol} (b)). We apply \prettyref{lem:shift_2}
to $T(z)=(zN(z^{-1})+A)^{-1}$ for $z\in\mathbb{C}$ with $\Re z\geq\eta$
and get that 
\begin{align*}
v & =\mathcal{L}_{\eta}^{\ast}\left(\overline{\left(\i m+\eta\right)N\left(\frac{1}{\i m+\eta}\right)+A}\right)^{-1}\mathcal{L}_{\eta}e^{\nu m}f\\
 & =\mathcal{L}_{\eta+\nu}^{\ast}\left(\overline{\left(\i m+\eta+\nu\right)N\left(\frac{1}{\i m+\eta+\nu}\right)+A}\right)^{-1}\mathcal{L}_{\eta+\nu}e^{\nu m}f.
\end{align*}
Using $\left(\i t+\eta+\nu\right)N\left(\frac{1}{\i t+\eta+\nu}\right)=(\i t+\eta)M\left(\frac{1}{\i t+\eta}\right)$
for $t\in\mathbb{R}$, we derive 
\[
e^{-\nu m}v=\mathcal{L}_{\eta}^{\ast}\left(\overline{\left(\i m+\eta\right)M\left(\frac{1}{\i m+\eta}\right)+A}\right)^{-1}\mathcal{L}_{\eta}f.
\]
Again, applying \prettyref{lem:shift_2} to $T(z)=\left(zM(z^{-1})+A\right)^{-1}$
for $z\in\mathbb{C}$ with $\Re z\geq\min\{\rho,\eta\}$ we get that
\begin{align*}
e^{-\nu m}v & =\mathcal{L}_{\eta}^{\ast}\left(\overline{\left(\i m+\eta\right)M\left(\frac{1}{\i m+\eta}\right)+A}\right)^{-1}\mathcal{L}_{\eta}f\\
 & =\mathcal{L}_{\rho}^{\ast}\left(\overline{\left(\i m+\rho\right)M\left(\frac{1}{\i m+\rho}\right)+A}\right)^{-1}\mathcal{L}_{\rho}f\\
 & =u,
\end{align*}
which gives $u\in H_{\eta-\nu}(\mathbb{R};H).$ Since $\eta\in]0,\rho+\nu[$
was chosen arbitrarily, we get the assertion.
\end{proof}

\section{Examples}

In this section we illustrate our results of the previous section
by three different types of differential equations which, however,
are all covered by the abstract notion of evolutionary equations.
We emphasize that we do not claim that in the forthcoming examples
the stability rates are optimal under the given constraints nor that
an exponential decay could not be obtained under lesser constraints.
But we emphasize that our approach provides a unified way to study
exponential stability of a broad class of differential equations.\\
We begin to study a class of differential-algebraic equations, where
the material law is of the simplest form. Moreover, we provide a strategy
of how to deal with initial values problems for this class. As a second
example we treat a partial differential-algebraic equation with finite
delay. We conclude this section with an example of a parabolic integro-differential
equation with an operator-valued kernel.

\subsection{Differential-algebraic equations of mixed type}

It turns out that in applications, the material law is often of the
form $M(\partial_{0,\rho}^{-1})=M_{0}+\partial_{0,\rho}^{-1}M_{1}$
(see for instance \cite{Picard,Picard_McGhee,Trostorff2012_NA}),
where $M_{0},M_{1}\in L(H)$. The corresponding evolutionary equation
is then of the form 
\begin{equation}
\left(\partial_{0,\rho}M_{0}+M_{1}+A\right)u=f,\label{eq:diff-alg-1}
\end{equation}

where $A:D(A)\subseteq H\to H$ is a maximal monotone linear operator.
In order to obtain the well-posedness of this evolutionary equation
we require that $M_{0}$ is selfadjoint and strictly positive definite
on its range, while $\Re M_{1}\coloneqq\frac{1}{2}(M_{1}+M_{1}^{\ast})$
is strictly positive definite on the kernel of $M_{0}$ (see \cite{Picard,Trostorff2012_NA,Trostorff2012_nonlin_bd}
for the proof of well-posedness). In order to obtain exponential stability
for this problem, we require that $\Re M_{1}$ is strictly positive
definite on the whole space $H.$ 
\begin{thm}
\label{thm:diff_alg}Let $A:D(A)\subseteq H\to H$ be a maximal monotone
linear operator and $M_{0},M_{1}\in L(H)$ such that $M_{0}$ is selfadjoint
and strictly positive definite on its range, and $\Re M_{1}\geq c>0$.
Then for each $\rho>0$ the solution operator $\left(\overline{\partial_{0,\rho}M_{0}+M_{1}+A}\right)^{-1}$
is exponentially stable with stability rate $\frac{c}{\|M_{0}\|}$.\end{thm}
\begin{proof}
We have to verify the hypotheses (a)-(c) for the function 
\[
M(z)=M_{0}+zM_{1}\quad(z\in\mathbb{C}).
\]
Obviously the assumption (a) holds. Let now $r>0$ and $\nu\geq0.$
Then we compute 
\[
(1-\nu z)M\left(z(1-\nu z)^{-1}\right)=(1-\nu z)M_{0}+zM_{1}
\]
for $z\in B_{\mathbb{C}}(r,r)\setminus\{\nu^{-1}\},$ which shows
(b). Let now $\nu\in]0,\frac{c}{\|M_{0}\|}[.$ In order to show (c)
we note, that for $z\in\mathbb{C}\setminus B_{\mathbb{C}}\left[-\frac{1}{2\nu},\frac{1}{2\nu}\right]$
there exists $t\in\mathbb{R}$ and $\rho>-\nu$ such that $z^{-1}=\i t+\rho.$
Thus, for $\rho\geq0$ we can estimate 
\[
\Re z^{-1}M(z)=\rho M_{0}+\Re M_{1}\geq c,
\]
while for $\rho\in\left]-\nu,0\right[$ we estimate 
\[
\Re z^{-1}M(z)=\rho M_{0}+\Re M_{1}\geq-\nu\|M_{0}\|+c>0.
\]
Thus, the assertion follows by \prettyref{thm:stability}
\end{proof}
To illustrate the versatility of our approach, we discuss the following
simple example of an evolutionary equation, whose type (elliptic,
parabolic or hyperbolic) changes on different parts of the underlying
domain. A similar example was also discussed in \cite{Picard2012_comprehensive_control,Picard2013_nonauto}.
\begin{example}
Let $\Omega\subseteq\mathbb{R}^{n}$ and $\Omega_{0},\Omega_{1}\subseteq\Omega$
be measurable disjoint subsets with positive Lebesgue measure. We
consider the evolutionary equation
\begin{equation}
\left(\partial_{0,\rho}\left(\begin{array}{cc}
\chi_{\Omega_{0}}+\chi_{\Omega_{1}} & 0\\
0 & \chi_{\Omega_{0}}
\end{array}\right)+\left(\begin{array}{cc}
c & 0\\
0 & c
\end{array}\right)+\left(\begin{array}{cc}
0 & \dive_{c}\\
\grad & 0
\end{array}\right)\right)\left(\begin{array}{c}
v\\
q
\end{array}\right)=\left(\begin{array}{c}
f\\
g
\end{array}\right),\label{eq:diff-alg}
\end{equation}
where $c>0$. The differential operator $\dive_{c}$ is defined as
the closure of the operator 
\begin{align*}
\dive|_{C_{c}^{\infty}(\Omega)^{n}}:C_{c}^{\infty}(\Omega)^{n}\subseteq L_{2}(\Omega)^{n} & \to L_{2}(\Omega)\\
\left(\phi_{i}\right)_{i\in\{1,\ldots,n\}} & \mapsto\sum_{i=1}^{n}\partial_{i}\phi_{i},
\end{align*}
where we denote by $\partial_{i}$ the partial derivative with respect
to the $i$-th coordinate. The operator $\grad$ is defined as the
negative adjoint of $\dive_{c},$ i.e. 
\[
\grad\coloneqq-\left(\dive_{c}\right)^{\ast}.
\]
This operator is just the usual weak gradient on $L_{2}(\Omega)$
with domain $H^{1}(\Omega).$ Note that then the operator matrix $\left(\begin{array}{cc}
0 & \dive_{c}\\
\grad & 0
\end{array}\right)$ is a skew-selfadjoint operator (and hence maximal monotone) on $H\coloneqq L_{2}(\Omega)\oplus L_{2}(\Omega)^{n}.$
Moreover, the operators 
\[
M_{0}=\left(\begin{array}{cc}
\chi_{\Omega_{0}}+\chi_{\Omega_{1}} & 0\\
0 & \chi_{\Omega_{0}}
\end{array}\right),\quad M_{1}=\left(\begin{array}{cc}
c & 0\\
0 & c
\end{array}\right)
\]
satisfy the assumptions of \prettyref{thm:diff_alg} and hence, the
solution operator is exponentially stable with stability rate $c.$
\\
Although this example seems to be quite easy, it seems hard to attack
the problem of solving \prettyref{eq:diff-alg-1} by using semigroup
techniques. The reason for that is that \prettyref{eq:diff-alg-1}
changes its type on different parts of the domain $\Omega.$ Indeed,
on $\Omega_{0}$ we obtain a hyperbolic problem of the form 
\[
\left(\partial_{0,\rho}\left(\begin{array}{cc}
1 & 0\\
0 & 1
\end{array}\right)+\left(\begin{array}{cc}
c & 0\\
0 & c
\end{array}\right)+\left(\begin{array}{cc}
0 & \dive_{c}\\
\grad & 0
\end{array}\right)\right)\left(\begin{array}{c}
v\\
q
\end{array}\right)=\left(\begin{array}{c}
f\\
g
\end{array}\right),
\]
while on $\Omega_{1}$ the problem becomes parabolic, namely 
\[
\left(\partial_{0,\rho}\left(\begin{array}{cc}
1 & 0\\
0 & 0
\end{array}\right)+\left(\begin{array}{cc}
c & 0\\
0 & c
\end{array}\right)+\left(\begin{array}{cc}
0 & \dive_{c}\\
\grad & 0
\end{array}\right)\right)\left(\begin{array}{c}
v\\
q
\end{array}\right)=\left(\begin{array}{c}
f\\
g
\end{array}\right),
\]
which yields, in case of $g=0$ the parabolic differential equation
\[
\partial_{0,\rho}v+cv-c^{-1}\dive_{c}\grad v=f.
\]
On the remaining part $\Omega\setminus(\Omega_{0}\cup\Omega_{1})$
the problem is elliptic
\[
\left(\left(\begin{array}{cc}
c & 0\\
0 & c
\end{array}\right)+\left(\begin{array}{cc}
0 & \dive_{c}\\
\grad & 0
\end{array}\right)\right)\left(\begin{array}{c}
v\\
q
\end{array}\right)=\left(\begin{array}{c}
f\\
g
\end{array}\right).
\]
Note that we can treat this problem, without requiring any explicit
transmission conditions on the interfaces $\partial\Omega_{0}$ and
$\partial\Omega_{1}$ and without imposing regularity assumptions
on these boundaries. 
\end{example}

\subsubsection*{Initial value problems}

Now, we present a possible way to tackle initial value problems for
equations of the form \prettyref{eq:diff-alg-1}. Consider the following
initial value problem%
\footnote{Note that it only makes sense to prescribe an initial value for $M_{0}u$
yielding an initial value for the part of $u$ lying in the range
of $M_{0}.$ %
} 
\begin{align}
\left(\partial_{0,\rho}M_{0}+M_{1}+A\right)u & =f\quad\mbox{ on }]0,\infty[\label{eq:ivp}\\
M_{0}u(0+) & =M_{0}u_{0}\nonumber 
\end{align}
for $M_{0},M_{1},A$ as before, $f\in H_{-\nu,0}(\mathbb{R};H)\cap H_{\rho,0}(\mathbb{R};H)$
for some $\nu,\rho>0$ with $\spt f\subseteq[0,\infty[$ and $u_{0}\in D(A).$
One way to deal with such problems is to consider the evolutionary
equation 
\[
\left(\partial_{0,\rho}M_{0}+M_{1}+A\right)\tilde{v}=f-\chi_{[0,\infty[}(m)M_{1}u_{0}-\chi_{[0,\infty[}(m)Au_{0}\quad\mbox{ on }\mathbb{R},
\]
for the new unknown $\tilde{v}\coloneqq u-\chi_{[0,\infty[}(m)u_{0}$
given by $\tilde{v}(t)=u(t)-\chi_{[0,\infty[}(t)u_{0}$ for almost
every $t\in\mathbb{R}$. Then the right-hand side belongs to $H_{\rho,0}(\mathbb{R};H)$
for positive $\rho,$ but does not decay if $f$ decays. Hence, this
approach can be used for the issue of well-posedness but it is not
appropriate for exponential stability. \\
Instead, we consider an alternative problem for the unknown 
\[
v\coloneqq u-\phi(m)u_{0},
\]
where $\phi$ is given by 
\[
\phi(t)\coloneqq\begin{cases}
1 & \mbox{ if }t\in[0,1],\\
2-t & \mbox{ if }t\in]1,2[,\\
0 & \mbox{ otherwise.}
\end{cases}
\]
It is clear that if $u$ satisfies \prettyref{eq:ivp}, then $v$
satisfies the equation 
\begin{equation}
\left(\partial_{0,\rho}M_{0}+M_{1}+A\right)v=f+\chi_{]1,2[}(m)M_{0}u_{0}-\phi(m)M_{1}u_{0}-\phi(m)Au_{0}\eqqcolon g\label{eq:ivp_modified}
\end{equation}
and vice versa. Since the function $\chi_{]1,2[}(m)M_{0}u_{0}-\phi(m)M_{1}u_{0}-\phi(m)Au_{0}$
belongs to $H_{\mu,0}(\mathbb{R};H)$ for every $\mu\in\mathbb{R}$,
we obtain $g\in H_{\rho,0}(\mathbb{R};H)\cap H_{-\nu,0}(\mathbb{R};H)$.
Thus, \prettyref{thm:diff_alg} applies to \prettyref{eq:ivp_modified}
and we get that \prettyref{eq:ivp_modified} is well-posed and $v\in H_{\mu,0}(\mathbb{R};H)$
for each $\mu\in]-\nu,\rho].$ This gives that $u\in H_{\mu,0}(\mathbb{R};H)$
for each $\mu\in]-\nu,\rho],$ since $\phi(m)u_{0}\in\bigcap_{\mu\in\mathbb{R}}H_{\mu,0}(\mathbb{R};H).$
It is left to show in which sense $M_{0}u$ attains the initial value
$M_{0}u_{0}.$ By \prettyref{eq:ivp_modified} we get that 
\[
\partial_{0,\rho}M_{0}v=g-M_{1}v-Av
\]
and the right hand side belongs to $H_{\rho,0}(\mathbb{R};H_{-1}(A+1)),$
where $H_{-1}(A+1)$ is the extrapolation space%
\footnote{For a boundedly invertible linear operator $C:D(C)\subseteq H\to H$
the extrapolation space $H_{-1}(C)$ is given as the completion of
$H$ with respect to the norm $|\cdot|_{H_{-1}(C)}$ defined as $|x|_{H_{-1}(C)}\coloneqq|C^{-1}x|_{H}$
for $x\in H$ (see \cite[Section 2.1]{Picard_McGhee}). %
} associated with $A+1$. Thus, a version of Sobolev's embedding theorem
(\cite[Lemma 3.1.59]{Picard_McGhee}, \cite[Lemma 5.2]{Kalauch2011})
yields that $M_{0}v$ is continuous as a function which attains values
in $H_{-1}(A+1).$ Furthermore, due to the causality of the solution
operator $\left(\overline{\partial_{0,\rho}M_{0}+M_{1}+A}\right)^{-1}$,
$v$ is supported on $[0,\infty[,$ since $\spt g\subseteq[0,\infty[.$
This yields $M_{0}v(0+)=M_{0}v(0-)=0$ and hence, since $\phi(0+)=1$
we get that $M_{0}u(0+)=M_{0}u_{0}$, where the equality holds in
$H_{-1}(A+1).$

\subsection{Linear partial differential equations with finite delay}

As a second example we study a differential equation with finite delay
of the form 
\begin{equation}
\left(\partial_{0,\rho}M_{0}+\tau_{h}+M_{1}+A\right)u=f,\label{eq:delay}
\end{equation}
where $M_{0},M_{1}\in L(H)$ such that $M_{0}$ is selfadjoint and
non-negative, $\Re M_{1}\geq c>1$, $A:D(A)\subseteq H\to H$ is linear
and maximal monotone and $\tau_{h}$ is the translation operator with
respect to time, i.e. $\left(\tau_{h}u\right)(t)=u(t+h)$ for $t\in\mathbb{R}$
and some $h\leq0.$ We will prove that under these assumptions, the
corresponding solution operator is exponentially stable and we give
an estimate for the stability rate. A similar problem is treated in
\cite[Example 4.14]{Batkai_2005} for a particular operator $A$,
where the well-posedness is shown via semigroups and a criterion for
the exponential stability is given, using the Spectral Mapping Theorem
for eventually norm continuous semigroups (cf. \cite[p. 280]{engel2000one}).

Before we state our stability result for \prettyref{eq:delay}, we
need to inspect the operator $\tau_{h}$ a bit closer.
\begin{lem}
Let $\rho,h\in\mathbb{R}.$ We define the operator 
\begin{align*}
\tau_{h}\colon H_{\rho,0}(\mathbb{R};H) & \to H_{\rho,0}(\mathbb{R};H)\\
u & \mapsto\left(t\mapsto u(t+h)\right).
\end{align*}
Then $\tau_{h}\in L(H_{\rho,0}(\mathbb{R};H))$ with $\|\tau_{h}\|=e^{\rho h}$
and $\tau_{h}=\mathcal{L}_{\rho}^{\ast}e^{\left(\i m+\rho\right)h}\mathcal{L}_{\rho}$.\end{lem}
\begin{proof}
Obviously, $\tau_{h}$ defines a bounded linear operator on $H_{\rho,0}(\mathbb{R};H).$
For $\phi\in C_{c}^{\infty}(\mathbb{R};H)$ we compute 
\begin{align*}
\mathcal{L}_{\rho}\left(\tau_{h}\phi\right)(t) & =\frac{1}{\sqrt{2\pi}}\intop_{\mathbb{R}}e^{-\i st}e^{-\rho s}\phi(s+h)\mbox{ d}s\\
 & =e^{\left(\i t+\rho\right)h}\frac{1}{\sqrt{2\pi}}\intop_{\mathbb{R}}e^{-\i st}e^{-\rho s}\phi(s)\mbox{ d}s\\
 & =e^{\left(\i t+\rho\right)h}\left(\mathcal{L}_{\rho}\phi\right)(t)
\end{align*}
for each $t\in\mathbb{R},$ which gives $\tau_{h}=\mathcal{L}_{\rho}^{\ast}e^{\left(\i m+\rho\right)h}\mathcal{L}_{\rho}$.
Moreover, by this unitary equivalence we get that 
\[
\|\tau_{h}\|=\|e^{\left(\i m+\rho\right)h}\|=e^{\rho h}.\tag*{\qedhere}
\]

\end{proof}
Using this lemma, we are able to write \prettyref{eq:delay} as an
evolutionary equation with 
\begin{equation}
M(z)=M_{0}+ze^{z^{-1}h}+zM_{1},\label{eq:material_law_delay}
\end{equation}
which is clearly analytic and bounded on balls of the form $B_{\mathbb{C}}(r,r)$
for $r>0$ if we require that $h\leq0.$ This restriction is natural,
since for $h\leq0$ the operator $\tau_{h}$ is forward causal, while
for $h>0$ it is backward causal. 
\begin{thm}
\label{thm:stability_delay}Let $A:D(A)\subseteq H\to H$ be a maximal
monotone linear operator, $M_{0},M_{1}\in L(H)$ such that $M_{0}$
is selfadjoint and non-negative and $\Re M_{1}\geq c>1.$ Moreover
let $h<0.$ Then for each $\rho>0$ the solution operator $\left(\overline{\partial_{0,\rho}M_{0}+\tau_{h}+M_{1}+A}\right)^{-1}$
is exponentially stable with stability rate $\nu_{0}>0$ satisfying
\[
\nu_{0}\|M_{0}\|+e^{-\nu_{0}h}=c.
\]
\end{thm}
\begin{proof}
We have to show that $M$ given by \prettyref{eq:material_law_delay}
satisfies the hypotheses (a)-(c) of Section 3. Obviously $M$ is analytic
on $\mathbb{C}\setminus\{0\},$ which shows (a). Let now $r>0$ and
$\nu\geq0$. As 
\[
(1-\nu z)M\left(z(1-\nu z)^{-1}\right)=(1-\nu z)M_{0}+ze^{(z^{-1}-\nu)h}+zM_{1}
\]
for $z\in B_{\mathbb{C}}(r,r)\setminus\{\nu^{-1}\}$, we see that
$z\mapsto(1-\nu z)M\left(z(1-\nu z)^{-1}\right)$ has an analytic
extension to $B_{\mathbb{C}}(r,r).$ Moreover we estimate 
\[
\sup_{z\in B_{\mathbb{C}}(r,r)}|e^{(z^{-1}-\nu)h}|=\sup_{\rho>\frac{1}{2r}}e^{(\rho-\nu)h},
\]
which is finite, since $h<0.$ Thus, hypothesis (b) is also satisfied.
For showing (c), let $\nu\in]0,\nu_{0}[$ and $z\in\mathbb{C}\setminus B_{\mathbb{C}}\left[-\frac{1}{2\nu},\frac{1}{2\nu}\right].$
Then there exists $\rho>-\nu$ and $t\in\mathbb{R}$ such that $z^{-1}=\i t+\rho.$
If $\rho\geq0$ we estimate 
\begin{align*}
\Re z^{-1}M(z) & =\rho M_{0}+e^{\rho h}\cos(th)+\Re M_{1}\geq c-1>0,
\end{align*}
and for the case $\rho\in]-\nu,0[$ we get 
\[
\Re z^{-1}M(z)=\rho M_{0}+e^{\rho h}\cos(th)+\Re M_{1}\geq-\nu\|M_{0}\|-e^{-\nu h}+c>0,
\]
since $\nu<\nu_{0}.$ This proves (c) and thus, the assertion follows
by \prettyref{thm:stability}.\end{proof}
\begin{rem}
Note that $M_{0}$ in \prettyref{eq:delay} is allowed to have a non-trivial
kernel which could also depend on the spatial variable (compare Subsection
4.1). Thus, \prettyref{thm:stability_delay} also covers a certain
class of differential-algebraic equations with delay.
\end{rem}

\subsection{Parabolic integro-differential equations}

We consider the following parabolic integro-differential equation
\begin{equation}
\partial_{0,\rho}u+Bu-C\ast Bu=f,\label{eq:orig-integro}
\end{equation}
where $B:D(B)\subseteq H\to H$ is linear such that $A\coloneqq B-c$
is maximal monotone for some $c>0$, $C:[0,\infty[\to L(H)$ is a
weakly measurable function such that $t\mapsto\|C(t)\|$ is measurable
and there exists $\nu_{0}>0$ with $\intop_{0}^{\infty}\|C(t)\|e^{\nu_{0}t}\mbox{ d}t<1$.
We set $U\coloneqq\{z\in\mathbb{C}\,|\,\Im z\leq\nu_{0}\}$ and define
the complex Fourier transform of $C$ by 
\[
\hat{C}(z)\coloneqq\frac{1}{\sqrt{2\pi}}\intop_{0}^{\infty}e^{-\i tz}C(t)\mbox{ d}t\quad(z\in U),
\]
where the integral is meant in the weak sense. Note that $\hat{C}:U\to L(H)$
is continuous and bounded on $U$ and analytic in the interior of
$U$%
\footnote{Here we use the fact that scalar analyticity and local boundedness
on a norming set yields analyticity (see \cite[Theorem 3.10.1]{hille1957functional}).%
}. The well-posedness and asymptotic behavior for equations of the
form \prettyref{eq:orig-integro}, including non-linear perturbations,
were discussed in several works (e.g. \cite{Cannarsa2003,Vlasov2010,Trostorff2012_integro}
and \cite{Cannarsa2011,Pruss2009,Berrimi2006,Cavalcanti2003,Trostorff2012_integro}
for a hyperbolic version of the problem), imposing additional constraints
on the kernel $C$. \\
Following \cite{Trostorff2012_integro} we are led to assume that
$C$ satisfies the following conditions:

\begin{enumerate}

\item $C(t)$ is selfadjoint for almost every $t\in\mathbb{R},$\item
$C(t)$ and $C(s)$ commute for almost every $t,s\in\mathbb{R},$\item
for all $t\in\mathbb{R}$ we have 
\begin{equation}
t\Im\hat{C}(t+\i\nu_{0})\leq0.\label{eq:pos_def_conv}
\end{equation}

\end{enumerate}
\begin{rem}
$\,$

\begin{enumerate}[(a)]

\item Note that \prettyref{eq:pos_def_conv} is equivalent to 
\[
\Im\hat{C}(t+\i\nu_{0})\leq0\quad(t\in]0,\infty[).
\]
\item  A typical example for a kernel satisfying the conditions above
is a real-valued, differentiable function $k:[0,\infty[\to[0,\infty[$
with $\intop_{0}^{\infty}k(t)e^{\nu_{0}t}\mbox{ d}t<1$ and $k'(t)\leq-k(t)\nu_{0}$
for every $t\geq0$. Similar kernels were considered by Pr�ss \cite{Pruss2009}
under a weaker constraint on $k'$. Indeed, the conditions 1. and
2. are trivially satisfied, since $k$ is real-valued. For showing
condition 3. we note that 
\[
e^{\nu_{0}t}k(t)-e^{\nu_{0}s}k(s)\leq\sup_{\xi\in[s,t]}e^{\nu_{0}\xi}(\nu_{0}k(\xi)+k'(\xi))\leq0,
\]
for every $t\geq s\geq0$. Thus, the function $t\mapsto e^{\nu_{0}t}k(t)$
is non-increasing and we estimate 
\begin{align*}
\Im\hat{k}(t+\i\nu_{0}) & =\frac{1}{\sqrt{2\pi}}\intop_{0}^{\infty}e^{\nu_{0}s}\sin(-ts)k(s)\mbox{ d}s\\
 & =\frac{1}{\sqrt{2\pi}}\sum_{k=0}^{\infty}\left(\intop_{k\frac{\pi}{t}}^{\left(2k+1\right)\frac{\pi}{t}}e^{\nu_{0}s}\sin(-ts)k(s)\mbox{ d}s+\intop_{\left(2k+1\right)\frac{\pi}{t}}^{2\left(k+1\right)\frac{\pi}{t}}e^{\nu_{0}s}\sin(-ts)k(s)\mbox{ d}s\right)\\
 & =\frac{1}{\sqrt{2\pi}}\sum_{k=0}^{\infty}\intop_{k\frac{\pi}{t}}^{(2k+1)\frac{\pi}{t}}\sin(-ts)\left(e^{\nu_{0}s}k(s)-e^{\nu_{0}(s+\frac{\pi}{t})}k\left(s+\frac{\pi}{t}\right)\right)\mbox{ d}s\\
 & \leq0
\end{align*}
for every $t\in]0,\infty[$ (compare \cite[Remark 3.6 (b)]{Trostorff2012_integro}). 

\item In \cite{Cannarsa2011} the authors considered real-valued
kernels $k:[0,\infty[\to\mathbb{R}$ such that $\intop_{0}^{\infty}k(s)e^{\nu_{0}s}\mbox{ d}s<1$
and the integrated kernel $[0,\infty[\ni t\mapsto\intop_{t}^{\infty}k(s)e^{\nu_{0}s}\mbox{ d}s$
gives rise to a positive definite convolution operator on $L_{2}([0,\infty[)$.
Again, the conditions 1. and 2. are satisfied, since $k$ is real-valued
and condition 3. holds according to \cite[Proposition 2.2 (a)]{Cannarsa2011}.

\end{enumerate}
\end{rem}
Before we can state a stability result for problems of the form \prettyref{eq:orig-integro},
we recall some properties of the convolution operator $C\ast.$
\begin{lem}
\label{lem:conv}We denote by $S(\mathbb{R};H)$ the space of simple
$H$-valued functions. Then for $\rho\geq-\nu_{0}$ the operator 
\begin{align*}
C\ast:S(\mathbb{R};H)\subseteq H_{\rho,0}(\mathbb{R};H) & \to H_{\rho,0}(\mathbb{R};H)\\
u & \mapsto\left(t\mapsto\intop_{\mathbb{R}}C(t-s)u(s)\,\mathrm{d}s\right)
\end{align*}
is bounded and linear with $\|C\ast\|_{L(H_{\rho,0}(\mathbb{R};H))}\leq\intop_{0}^{\infty}\|C(t)\|e^{\nu_{0}t}\,\mathrm{d}t$
and can therefore be extended to $H_{\rho,0}(\mathbb{R};H).$ Moreover,
for $u\in H_{\rho,0}(\mathbb{R};H)$, $\rho\geq-\nu_{0}$ we have
\begin{equation}
\left(\mathcal{L}_{\rho}(C\ast u)\right)(t)=\sqrt{2\pi}\hat{C}(t-\i\rho)\left(\mathcal{L}_{\rho}u(t)\right)\label{eq:spec_conv}
\end{equation}
for almost every $t\in\mathbb{R}$.\end{lem}
\begin{proof}
A proof for the first assertion can be found in \cite[Lemma 3.1]{Trostorff2012_integro}.
The proof of formula \prettyref{eq:spec_conv} is straight forward
and we omit it.
\end{proof}
According to \prettyref{lem:conv}, the operator $(1-C\ast)$ is boundedly
invertible in $H_{\rho,0}(\mathbb{R};H)$ for each $\rho\geq-\nu_{0}.$
Therefore, instead of considering \prettyref{eq:orig-integro} we
can study 
\begin{equation}
\left(\partial_{0,\rho}(1-C\ast)^{-1}+B\right)u=(1-C\ast)^{-1}f,\label{eq:modi_integro}
\end{equation}
or equivalently 
\[
\left(\partial_{0,\rho}(1-C\ast)^{-1}+c+A\right)u=(1-C\ast)^{-1}f.
\]
Note, that this as an evolutionary equation of the form \prettyref{eq:evo}
where $M$ is defined by 
\begin{equation}
M(z)=(1-\sqrt{2\pi}\hat{C}(-\i z^{-1}))^{-1}+cz,\quad\left(z\in\mathbb{C}\setminus B_{\mathbb{C}}\left(-\frac{1}{2\nu_{0}},\frac{1}{2\nu_{0}}\right)\right)\label{eq:material_law_integro}
\end{equation}
where we have used \prettyref{lem:conv}. The next lemma shows that
\prettyref{eq:pos_def_conv} already implies that the same condition
holds if one replaces $-\nu_{0}$ by $\rho$ for arbitrary $\rho\geq-\nu_{0}.$
\begin{lem}
\label{lem:one_nu_every_nu}Assume that $C$ satisfies the conditions
1., 2. and 3. above. Then for every $\rho\geq-\nu_{0}$ we have 
\[
t\Im\hat{C}(t-\i\rho)\leq0.
\]
\end{lem}
\begin{proof}
The proof can be done analogously to the one of \cite[Lemma 3.7]{Trostorff2012_integro}.
\end{proof}
We now state our stability result for \prettyref{eq:modi_integro}.
\begin{thm}
\label{thm:decay_integro}Let $A:D(A)\subseteq H\to H$ be maximal
monotone and linear and let $c>0$. Moreover, let $C:[0,\infty[\to L(H)$
be weakly measurable, such that $t\mapsto\|C(t)\|$ is measurable
and there exists $\nu_{0}>0$ such that $\intop_{0}^{\infty}\|C(t)\|e^{\nu_{0}t}<1$
and $C$ satisfies the conditions 1., 2. and 3. from above. Then for
each $\rho>0$ the solution operator $\left(\overline{\partial_{0,\rho}(1-C\ast)^{-1}+c+A}\right)^{-1}$
exists and is exponentially stable with a stability rate $\nu_{1}\in]0,\nu_{0}]$
satisfying
\[
\nu_{1}\left(1-\intop_{0}^{\infty}\|C(s)\|e^{\nu_{1}s}\mbox{ d}s\right)^{-1}\leq c.
\]
\end{thm}
\begin{proof}
Let $\nu_{1}\in]0,\nu_{0}]$ such that $\nu_{1}\left(1-\intop_{0}^{\infty}\|C(s)\|e^{\nu_{1}s}\mbox{ d}s\right)^{-1}\leq c.$
We prove that $M$ given by \prettyref{eq:material_law_integro} satisfies
the hypotheses of Section 3. The assumption (a) is clear, since \foreignlanguage{english}{$\|\sqrt{2\pi}\hat{C}(-\i z^{-1})\|<1$}
for each $z\in\mathbb{C}\setminus B_{\mathbb{C}}\left[-\frac{1}{2\nu_{1}},\frac{1}{2\nu_{1}}\right]$.
Let now $r>0$ and $0\leq\nu<\nu_{1}.$ Then for $z\in B_{\mathbb{C}}(r,r)\setminus\{\nu^{-1}\}$
we compute 
\begin{align*}
(1-\nu z)M(z(1-\nu z)^{-1}) & =(1-\nu z)\left(\left(1-\sqrt{2\pi}\hat{C}\left(-\i z^{-1}(1-\nu z)\right)\right)^{-1}+cz(1-\nu z)^{-1}\right)\\
 & =(1-\nu z)\left(1-\sqrt{2\pi}\hat{C}\left(-\i(z^{-1}-\nu)\right)\right)^{-1}+cz,
\end{align*}
which has a holomorphic extension in $\nu^{-1}.$ Noting that for
each $z\in B_{\mathbb{C}}(r,r)$ we have that $z^{-1}=\i t+\rho$
for some $\rho>\frac{1}{2r},t\in\mathbb{R},$ we estimate 
\begin{align*}
\|\sqrt{2\pi}\hat{C}(-\i(z^{-1}-\nu))\| & \leq\intop_{0}^{\infty}e^{-\left(\rho-\nu\right)s}\|C(s)\|\mbox{ d}s\\
 & \leq\intop_{0}^{\infty}e^{\nu_{0}s}\|C(s)\|\mbox{ d}s<1.
\end{align*}
Hence, the extension of $(1-\nu z)M(z(1-\nu z)^{-1})$ to $B_{\mathbb{C}}(r,r)$
is indeed bounded for each $r>0,\nu\in[0,\nu_{1}[.$ We now show that
$M$ satisfies the assumption (c) on $\mathbb{C}\setminus B_{\mathbb{C}}\left[-\frac{1}{2\nu_{1}},\frac{1}{2\nu_{1}}\right]$.
We follow the strategy of the proof of \cite[Lemma 3.8]{Trostorff2012_integro}.
Let $z\in\mathbb{C}\setminus B_{\mathbb{C}}\left[-\frac{1}{2\nu_{1}},\frac{1}{2\nu_{1}}\right]$.
Note that then there exists $\rho>-\nu_{1}$ and $t\in\mathbb{R}$
such that $z^{-1}=\i t+\rho.$ We set $D\coloneqq|1-\sqrt{2\pi}\hat{C}(t-\i\rho)|^{-1}$,
which is well-defined according to \prettyref{lem:one_nu_every_nu}.
Note that 
\[
(1-\sqrt{2\pi}\hat{C}(t-\i\rho))^{-1}=D^{2}(1-\sqrt{2\pi}\hat{C}(-t-\i\rho)),
\]
where we have used assumption 1. Moreover, due to assumption 2., we
get that 
\[
D^{2}(1-\sqrt{2\pi}\hat{C}(-t-\i\rho))=D(1-\sqrt{2\pi}\hat{C}(-t-\i\rho))D.
\]
Thus, we obtain for $x\in H$ 
\begin{align*}
\Re\langle z^{-1}M(z)x|x\rangle & =\Re\left\langle \left.z^{-1}\left((1-\sqrt{2\pi}\hat{C}(-\i z^{-1}))^{-1}+cz\right)x\right|x\right\rangle \\
 & =\Re\left\langle \left.\left(\rho\Re\left(1-\sqrt{2\pi}\hat{C}(-t-\i\rho)\right)+\sqrt{2\pi}t\Im\hat{C}(-t-\i\rho)\right)Dx\right|Dx\right\rangle +c|x|^{2}\\
 & \geq\rho\Re\left\langle \left.D\left(1-\sqrt{2\pi}\hat{C}(-t-\i\rho)\right)Dx\right|x\right\rangle +c|x|^{2}.
\end{align*}
If $\rho$ is non-negative, the latter term can be estimated by $c$.
For negative $\rho$ we observe that 
\begin{align*}
\left\Vert D\Re\left(1-\sqrt{2\pi}\hat{C}(-t-\i\rho)\right)D\right\Vert  & \leq\left\Vert \left(1-\sqrt{2\pi}\hat{C}(-t-\i\rho)\right)^{-1}\right\Vert \\
 & \leq\frac{1}{1-\|\sqrt{2\pi}\hat{C}(-t-\i\rho)\|}\\
 & \leq\left(1-\intop_{0}^{\infty}e^{\nu_{1}s}\|C(s)\|\mbox{ d}s\right)^{-1}
\end{align*}
and hence, 
\begin{align*}
\rho\Re\left\langle \left.D\left(1-\sqrt{2\pi}\hat{C}(-t-\i\rho)\right)Dx\right|x\right\rangle +c|x|^{2} & \geq\left(\rho\left(1-\intop_{0}^{\infty}e^{\nu_{1}s}\|C(s)\|\mbox{ d}s\right)^{-1}+c\right)|x|^{2}\\
 & >\left(-\nu_{1}\left(1-\intop_{0}^{\infty}e^{\nu_{1}s}\|C(s)\|\mbox{ d}s\right)^{-1}+c\right)|x|^{2},
\end{align*}
which shows that $M$ satisfies hypothesis (c), according to the choice
of $\nu_{1}$. Thus, the assertion follows by \prettyref{thm:stability}.\end{proof}
\begin{rem}
\prettyref{thm:decay_integro} gives the exponential stability for
equation \prettyref{eq:modi_integro}. This also yields the exponential
stability of the original problem \prettyref{eq:orig-integro}, since
the operator $(1-C\ast)^{-1}$ leaves the space $H_{-\nu,0}(\mathbb{R};H)$
for all $\nu\leq\nu_{0}$ invariant. Indeed, observing that 
\[
e^{\nu m}(1-C\ast)^{-1}=\left(1-\left(e^{\nu m}C\right)\ast\right)^{-1}e^{\nu m}
\]
we obtain for $f\in H_{-\nu,0}(\mathbb{R};H)\cap H_{\rho,0}(\mathbb{R};H)$
\[
e^{\nu m}(1-C\ast)^{-1}f=\left(1-\left(e^{\nu m}C\right)\ast\right)^{-1}e^{\nu m}f\in L_{2}(\mathbb{R};H).
\]
\end{rem}

\end{document}